\documentclass[a4paper,11pt]{amsart}

\usepackage[utf8]{inputenc}
\usepackage[english]{babel}
\usepackage{mathtools,amsthm,amsmath,amssymb,amsfonts,amscd,tabu}
\usepackage[numbers]{natbib}
\usepackage{mathrsfs}

\usepackage{graphicx}
\usepackage{layaureo}
\usepackage{color}

\usepackage{enumerate}
\usepackage{hyperref}

\usepackage[new]{old-arrows}
\usepackage{rotating}

\usepackage[small,nohug,heads=vee]{diagrams}
\diagramstyle[labelstyle=\scriptstyle]
\usepackage[all,cmtip]{xy}

\usepackage{tikz-cd}

\theoremstyle{plain}
\newtheorem{thm}{Theorem}[section]
\newtheorem{theorem}{Theorem}[section]
\newtheorem*{thm*}{Theorem}
\newtheorem{proposition}[thm]{Proposition}
\newtheorem{lemma}[thm]{Lemma}
\newtheorem{corollary}[thm]{Corollary}

\theoremstyle{definition}

\theoremstyle{definition}

\newtheorem{remark}{Remark}
\theoremstyle{definition}
\newtheorem{example}{Example}

\makeatletter
\newsavebox\myboxA
\newsavebox\myboxB
\newlength\mylenA

\newcommand{\N}{\mathbb N}
\newcommand{\Z}{\mathbb Z}
\newcommand{\R}{\mathbb R}
\newcommand{\C}{\mathbb C}

\newcommand{\nospaceperiod}{\makebox[0pt][l]{\,.}}
\newcommand{\nospacecomma}{\makebox[0pt][l]{\,,}}

\DeclareMathOperator{\Spec}{Spec}

\DeclareMathOperator{\Frac}{Frac}
\DeclareMathOperator{\id}{id}

\DeclareMathOperator{\Gr}{Gr}

\DeclareMathOperator{\Res}{Res}

\DeclareMathOperator{\Ker}{Ker}
\DeclareMathOperator{\Coker}{Coker}
\DeclareMathOperator{\IIm}{Im}

\DeclareMathOperator{\HH}{H}
\DeclareMathOperator{\HHH}{\mathbb{H}}

\DeclareMathOperator{\Tot}{Tot}

\DeclareMathOperator{\Cone}{Cone}

\title{An Algebraic Description of the Monodromy of Log Curves}
\author{Pietro Gatti}
\address{KU Leuven, 
Celestijnenlaan 200B, B-3001 Leuven, Belgium}
\address{University of Padua, Torre Archimede, Via Trieste 63, 35121 Padova, Italy}
\email{pietro.gatti@kuleuven.be}

\subjclass[2010]{16F40; 16D05}
\keywords{log curves; monodromy; semistable degenerations; invariant cycles}
\begin{document}
\maketitle

\begin{abstract}
Let \(k\) be an algebraically closed field of characteristic \(0\). For a log curve \(X/k^{\times}\) over the standard log point (\cite{Kato00}), we define (algebraically) a combinatorial monodromy operator on its log-de Rham cohomology group. The invariant part of this action has a cohomological description, it is the Du Bois cohomology of \(X\) (\cite{DuBois81}). This can be seen as an analogue of the invariant cycles exact sequence for a semistable family (as in the complex, \'etale and \(p\)-adic settings). In the specific case in which \(k=\C\) and \(X\) is the central fiber of a semistable degeneration over the complex disc, our construction recovers the topological monodromy and the classical local invariant cycles theorem. In particular, our description allows an explicit computation of the monodromy operator in this setting.\\

\end{abstract}

\section{Introduction}\label{intro}
Let \(S/\C\) be a smooth connected affine curve. Suppose that \(\pi:\mathfrak X\to S\) is a proper, separated morphism of finite type from a smooth scheme \(\mathfrak X/\C\). We will denote with \(\mathfrak X_s\) the fiber of \(\pi\) above any closed point \(s\in S\). We assume that \(\mathfrak X_s\) is smooth for \(s\in S'=S\setminus \{P\}\) so that, taking analytifications, the restriction of \(\pi^{an}\) to \((S')^{an}\) is a locally trivial \(C^\infty\)-fibration. In particular, if one takes a small punctured disk \(\Delta^*\subset S'\) centered at \(P\) and a point \(s\in \Delta^*\), there is an action of the fundamental group \(\pi_1(\Delta^*;s)\) on the singular cohomology groups \(\HH^i(\mathfrak X_s^{an};\C)\), for any integer \(i\). The monodromy operator \(T\) is the endomorphism induced on \(\HH^i(\mathfrak X_s^{an};\C)\) by the positive generator of \(\pi_1(\Delta^*;s)\simeq\Z\). The Local Monodromy Theorem asserts that \(T\) is quasi-unipotent. More precisely, if one assumes that the special fiber \(X=\mathfrak X_P\) is a simple normal crossing divisor, then \(T\) is unipotent. From now one we will consider this case and put \(N=-1/(2\pi i) \log(T)\) for the nilpotent version of the monodromy. For any integer \(i\), one has an invariant cycles exact sequence (\cite[§1.7]{DeCataldo09}, \cite[Corollaire 6.2.8]{Beilinson82})
\begin{equation}\label{invcycesintro}
\HH^i(X^{an};\C)\overset{sp}{\longrightarrow}\HH^i(\mathfrak X^{an}_s;\C)\overset{N}{\longrightarrow}\HH^i(\mathfrak X^{an}_s;\C),
\end{equation}
where \(sp\) comes from the deformation retract of \(\mathfrak X^{an}\) onto \(X^{an}\). Moreover we have injectivity of \(sp\) for \(i=1\).\\
In \cite{Steenbrink76}, the monodromy has been linked to the ``limit Hodge structure'', interpreting the cohomology of the generic fiber via a complex of logarithmic differentials \(\Omega_{\mathfrak X/S}^\bullet(\log X)\). In particular, the hypercohomology of \(\Omega_{\mathfrak X/S}^\bullet(\log X)\otimes \mathcal O_X\), the pull back to the special fiber, coincides with \(\HH^i(\mathfrak X^{an}_s;\C)\). If we put on \(X\) the log structure induced by the embedding \(X\hookrightarrow \mathfrak X\), we obtain a log smooth scheme over the standard log point \(\C^{\times}\) (we refer to \cite{Kato89} for the definitions). Its relative log-de Rham complex \(\omega^\bullet_{X/\C^{\times}}\) is isomorphic to the complex \(\Omega_{\mathfrak X/S}^\bullet(\log X)\otimes \mathcal O_X\) and therefore computes the cohomology of the generic fiber:
\begin{equation}\label{compintro}
\HH^i(\mathfrak X^{an}_s;\C)\overset{\sim}{\longrightarrow}\HHH^i(X,\omega^\bullet_{X/\C^{\times}}).
\end{equation}
In this paper we will consider the case of a semistable family of curves, that is \(\dim X=1\). We compute explicitly the cohomology of \(\omega^\bullet_{X/\C^{\times}}\) in terms of the irreducible components of \(X\) and the combinatorial data of their intersections, without any reference to the family. This allows us to construct a combinatorial monodromy \(\tilde{N}\) acting on \(\HHH^1(X,\omega^\bullet_{X/\C^{\times}})\), obtaining a new description for the topological monodromy.
\begin{theorem}\label{th1}
The combinatorial monodromy \(\tilde{N}\) and the topological monodromy \(N\), described above, coincide via the isomorphism in (\ref{compintro}).
\end{theorem}
As we said, we define a monodromy  \(\tilde{N}\) in a slightly more general setting than that of a semistable family of curves over the complex disc. This extension goes in two directions. Firstly, our method is completely algebraic, so we work over an algebraically closed field of characteristic \(0\), indicated with \(k\). This is not new: it has been known for a long time that the monodromy is related to an algebraic differential equation, the Gauss-Manin conncetion. Secondly, our construction does not involve the ``surrounding'' of \(X\), i.e. we do not suppose that \(X\) is the special fiber of a family. More precisely, we construct \(\tilde{N}\) for a log curve \(X\) in the sense of \cite{Kato00} (see the beginning of §\ref{logcurve}). This generalizes the case of a semistable degeneration of curves. For example, we have monodromy for a logarithmic embedding (cf. \cite{Kato96}). In particular, for a semistable degeneration over the disc, this shows that the monodromy is endoced in the central fiber when considered as a log scheme. \\
The ideas used in this paper are fundamentally inspired by the \(p\)-adic analogue of this theory: the study of the Hyodo-Kato monodromy (cf. \cite{Hyodo94,Coleman99,Coleman10}) for a semistable curve over a discrete valuation ring of mixed characteristic. Specifically, we reproduce the description of the Hyodo-Kato monodromy presented in \cite{Coleman99} in terms of residues and the dual graph of \(X\), where the log structure replaces the use of rigid analytic tubes. In addition, in \cite{Chiarellotto99} and \cite{Chiarellotto16}, \(p\)-adic invariant cycles exact sequences have been introduced, reinforcing the idea that the invariant part of the monodromy of a family should have a cohomological description in terms of the special fiber. In this direction, we have the following result, whose proof is essentially a combinatorial argument in the spirit of \cite{Chiarellotto16}.
\begin{theorem}\label{th2}
Let \(X/k^{\times}\) be a proper log curve over the standard log point. We assume that \(X\) has no marked points and that its irreducible components are smooth (cf. \cite{Kato00}). Then we have a short exact sequence
\[
0\longrightarrow \HH^1_{DB}(X/k)\longrightarrow\HHH^1(X,\omega^\bullet_{X/k^{\times}})\overset{\tilde{N}}{\longrightarrow}\HHH^1(X,\omega^\bullet_{X/k^{\times}}),
\]
where \(\HH^1_{DB}(X/k)=\HHH^1(X,\tilde{\Omega}_{X/k}^\bullet)\) denotes the Du Bois cohomology of \(X\), as in \cite[p.69]{DuBois81}.
\end{theorem}
When \(k=\C\) and \(X\) comes from a family, this result recovers the classical invariant cycles sequence (\ref{invcycesintro}), indeed there is an isomorphism \(\HH^1_{DB}(X/\C)\simeq \HH^1(X;\C)\), as shown in \cite[Th\'eor\`eme 4.5]{DuBois81}.\\
The content of this paper is structured as follows. After setting some conventions in §\ref{conv}, in §\ref{logcurve} we explain what are the properties of the log curve \(X/k^{\times}\) we study. We will obtain an explicit computation of the log differentials \(\omega^1_{X/k^{\times}}\) and those of the irreducible components of \(X\). In §\ref{Laurent} we study an algebraic definition of residues for curves in the \'etale setting. In §\ref{res} we construct a double complex \(\mathcal A_{X/k^{\times}}\) from the components of \(X\) and its dual graph. We show that this complex computes the hypercohomology \(\HHH^1(X,\omega^{\bullet}_{X/k^{\times}})\). This allows us to represent cohomology classes of \(\HHH^1(X,\omega^{\bullet}_{X/k^{\times}})\) in terms of global sections. In §\ref{mon}, this concrete description will give us the possibility to define the combinatorial monodromy \(\tilde{N}\). Then we show that the kernel of \(\tilde{N}\) is given by the Du Bois cohomology of \(X\), proving Theorem \ref{th2}. In §\ref{mongm} we review the theory in the complex setting. We recall the construction of the Gauss-Manin connection and collect the classical results that allow us to compute the monodromy algebraically on the special fiber. Section §\ref{comp} is devoted to the proof of Theorem \ref{th1}.\\\\
We plan to extend this framework to the situation of de Rham cohomology with coefficients i.e. integrable connections. Inspired by the similarities between our combinatorial description of the monodromy and the \(p\)-adic one appearing in \cite{Coleman10} and \cite{Chiarellotto16}, we expect to obtain results similar to those established in the arithmetic setting. In particular, a sufficient condition for the non-exactness of (\ref{invcycesintro}) when we have unipotent coefficients (cf. \cite[Theorem 10]{Chiarellotto16}). Since unipotent connections are not semi-simple, this goes in a different direction then the one developed in the classical and \(\ell\)-adic settings in \cite{Beilinson82}.

\section*{acknowledgements}
The author would like to express his gratitude to Bruno Chiarellotto for the valuable comments, corrections and advices on this paper. Thanks are also due to Nicola Mazzari for the helpful discussions concerning this research. The author was partially funded by Nero Budur's research project G0B2115N from the Research Foundation of Flanders.

\section{Notation and conventions}\label{conv}
Unless otherwise stated, all sheaves on a scheme are considered as sheaves on the small \'etale site. For a log scheme \(X\), we denote by \(X^\circ\), \(\mathcal O_X\) and \(\alpha_X:\mathcal M_X\to\mathcal O_X\) the underlying scheme, the structure sheaf and the log structure on \(X\), respectively. The \(i\)-th cohomology of a sheaf \(\mathcal F\) (resp. the \(i\)-th hypercohomology of a complex of sheaves \(\mathcal F^\bullet\)) on \(X^\circ\), will be indicated with \(\HH^i(X,\mathcal F)\) (resp. with \(\HHH^i(X,\mathcal F^{\bullet})\)). For a morphism of log schemes \(f:X\to Y\), we denote by \(f^\circ:X^\circ \to Y^\circ\) the underlying morphism of schemes. We denote its sheaf of relative log differentials with \(\omega^1_{X/Y}\), the associated de Rham complex with \(\omega^\bullet_{X/Y}\) and we put \(\HH^i_{\log}(X/Y):=\HHH^i(X,\omega^\bullet_{X/Y})\). Since no confusion will arise, we will use the notation \(\Omega^\bullet_{X/Y}\) instead of the more precise \(\Omega^\bullet_{X^\circ/Y^\circ}\) for the scheme-theoretic algebraic de Rham complex. The usual de Rham cohomology groups will be \(\HH^i_{dR}(X/Y):=\HHH^i(X,\Omega^{\bullet}_{X/Y})\).\\
Given a complex \(C^{\bullet}\) with differential \(d^i:C^i\to C^{i+1}\) in position \(i\) and an integer \(k\), we will denote with \(C^{\bullet+k}\) the complex with differential \(d^{i+k}:C^{i+k}\to C^{i+k+1}\) in position \(i\) and with \(C^{\bullet}[k]\) the complex with differential \((-1)^kd^{i+k}:C^{i+k}\to C^{i+k+1}\) in position \(i\).

\section{Log-curve and its components}\label{logcurve}
Let \(k\) be an algebraically closed field of characteristic \(0\) and let \(k^\times=(\Spec k,\N\oplus k^*)\) be the standard log point. We consider a proper log curve \(X\) over \(k^\times\) in the sense of F.~Kato (\cite{Kato00}). We add the following assumptions on \(X\):
\begin{enumerate}[(i)]
\item The log structure of \(X\) does not have marked points. In the terminology of \cite{Kato00} this means \(\overline{\mathcal M}_{X,x}\simeq 0\) or \(\Z\), for any point \(x\).
\item The irreducible components of \(X^\circ\) are smooth.
\end{enumerate}
\begin{example} The curve \(X^\circ\) is a simple normal crossing divisor in some smooth variety \(\mathfrak X\), the log structure \(\mathcal M_X\) is the pullback of the divisorial log structure on \(\mathfrak X\) induced by \(X^\circ\) (cf. \cite[§1.5]{Kato89}).
\end{example}
Using \cite[Theorem 1.3]{Kato00} and its proof, we deduce that \(X^{\circ}\) has at most ordinary double points, and we obtain the following description for its log structure. \'Etale locally around a double point we have a factorization
\[
X^\circ\longrightarrow\Spec k[x,y]/(xy)\longrightarrow\Spec k
\]
where the first arrow is \'etale and the morphism \(f:X\rightarrow k^\times\) admits a chart characterised by
\[
\begin{array}{ccc}
\mathbb N^2\oplus_{\Delta, \N, m}\N\longrightarrow \mathcal O_{X}, \quad&\N\longrightarrow k, \quad& \N\longrightarrow \N^2\oplus_{\Delta, \N, m}\N\nospaceperiod \\
((1,0),0)\longmapsto x \quad&1\longmapsto 0 \quad&1\longmapsto ((0,0),1) \\
((0,1),0)\longmapsto y\quad\\
((0,0),1)\longmapsto 0\quad
\end{array}
\]
Here \(\mathbb N^2\oplus_{\Delta, \N, m}\N\) denotes the amalgamated sum of the diagonal \(\Delta:\N\rightarrow \N^2\) and the multiplication by a natural number \(m:\N\to\N\).\\
\'Etale locally around a smooth point we have a factorization
\[
X^\circ\longrightarrow\Spec k[z]\longrightarrow\Spec k,
\]
where the first arrow is \'etale, and \(f\) admits a chart obtained from
\[
\begin{array}{ccc}
\mathbb N\longrightarrow \mathcal O_{X},\quad &\N\longrightarrow k,\quad & \N\longrightarrow \N\nospaceperiod \\
1\longmapsto 0 \quad&1\longmapsto 0 \quad&1\longmapsto 1 \\
\end{array}
\]

We attach to \(X\) its dual graph \(\Gr(X)=(\mathscr V, \mathscr E)\) in the following way. Each vertex \(v\) in \(\mathscr V\) corresponds to an irreducible component \(X_v^\circ\) of \(X^\circ\). To an intersection point \(X_e^\circ\) in \(X_v^\circ\cap X_w^\circ\) we associate an oriented edge \(e=[v,w]\) in \(\mathscr E\). In contrast with the construction of the dual graph in \cite{Coleman99} and \cite{Chiarellotto16}, each point of intersection appears only once in the graph and the corresponding edge has an arbitrary orientation.\\
Let \(e=[v,w]\) be an edge in \(\Gr(X\)), we adopt the following convention for the choice of \'etale coordinates for \(X^{\circ}_v\) and \(X^{\circ}_w\) around \(X^{\circ}_e\): the coordinate \(x\) is on \(X^{\circ}_v\) and the coordinate \(y\) is on \(X^{\circ}_w\). More precisely, we have a factorization
\begin{equation}\label{coordinates}
\begin{tikzcd}
X^{\circ}_e\arrow[r] \arrow[hookrightarrow]{d} &\Spec k \arrow[hookrightarrow]{d}{x\mapsto 0} \arrow[r] &\Spec k \arrow[d,equal]\\
X^{\circ}_v\arrow[r] \arrow[hookrightarrow]{d} &\Spec k[x] \arrow[hookrightarrow]{d} \arrow[r] &\Spec k \arrow[d,equal]\\
X^{\circ}\arrow[r] &\Spec k[x,y]/(xy)\arrow[r] &\Spec k\\
X^{\circ}_w\arrow[r] \arrow[hookrightarrow]{u} &\Spec k[y] \arrow[hookrightarrow]{u} \arrow[r] &\Spec k \arrow[u,equal]\\
X^{\circ}_e\arrow[r] \arrow[hookrightarrow]{u} &\Spec k \arrow[hookrightarrow]{u}{y\mapsto 0} \arrow[r] &\Spec k \arrow[u,equal]\nospacecomma
\end{tikzcd}
\end{equation}
where the horizontal arrows on the left are \'etale. For each \(v\) in \(\mathscr V\) and \(e\) in \(\mathscr E\) we will denote the closed embeddings of components and points of interscetion respectively by \(i^{\circ}_v:X^{\circ}_v\longrightarrow X^{\circ}\) and \(i^{\circ}_e:X^{\circ}_e\longrightarrow X^{\circ}\). We obtain log schemes \(X_v=(X^{\circ}_v,\mathcal M_{X_v})\), where \(\mathcal M_{X_v}\) is the pullback of \(\mathcal M_X\) via \(i^{\circ}_v\) as a log structure. From the charts for \(f:X\to k^\times\), one deduces local charts for \(f_v:X_v\to k^\times\) as follows. We assume \(e=[v,w]\), the case \(e=[w,v]\) being symmetric. \'Etale locally around \(X^{\circ}_e\), a chart for \(f_v\) is deduced from
\[
\begin{array}{ccc}
\mathbb N^2\oplus_{\Delta, \N, m}\N\longrightarrow \mathcal O_{X_v},\quad &\N\longrightarrow k,\quad & \N\longrightarrow \mathbb N^2\oplus_{\Delta, \N, m}\N\nospaceperiod \\
((1,0),0)\longmapsto x \quad&1\longmapsto 0 \quad&1\longmapsto ((0,0),1) \\
((0,1),0)\longmapsto 0\quad\\
((0,0),1)\longmapsto 0\quad
\end{array}
\]
\'Etale locally around a smooth point of \(X^{\circ}_v\), a chart for \(f_v\) is obtained from
\[
\begin{array}{ccc}
\mathbb N\longrightarrow \mathcal O_{X_v}, \quad&\N\longrightarrow k, \quad& \N\longrightarrow \N\nospaceperiod \\
1\longmapsto 0 \quad&1\longmapsto 0 \quad&1\longmapsto 1
\end{array}
\]
Similarly, for each edge \(e\) in \(\mathscr E\) we obtain log points \(X_e=(X^{\circ}_e,\mathcal M_{X_e})\), where \(\mathcal M_{X_e}\) is the pullback of \(\mathcal M_X\) via \(i_e\) as a log structure. A chart for \(X_e\to k^\times\) is given by
\[
\begin{array}{ccc}
\mathbb N^2\oplus_{\Delta, \N, m}\N\longrightarrow \mathcal O_{X_e}=k, \quad&\N\longrightarrow k, \quad& \N\longrightarrow \mathbb N^2\oplus_{\Delta, \N, m}\N \\
((1,0),0)\longmapsto 0 \quad&1\longmapsto 0 \quad&1\longmapsto ((0,0),1)\\
((0,1),0)\longmapsto 0\quad\\
((0,0),1)\longmapsto 0\quad
\end{array}
\]
We will now use the local description of the log structures to compute the sheaves of log differentials for the log schemes we have introduced. 
Around double points
\[
\omega_{X/k^\times}^1\simeq \frac{\mathcal O_X d\log x + \mathcal O_X d\log y}{\mathcal O_X (d\log x+d\log y)},
\]
where \(xd\log x=dx\) and \(yd\log y=dy\). Notice that the multiplication by \(m\), showing up in the charts, disappears. Indeed, the log differential \(d\log 0\), coming from \(((0,0),1)\), satisfies \(md\log 0=d\log x +d\log y\), but \(m\) is invertible in \(k\).\\
Around smooth points
\[
\omega_{X/k^\times}^1\simeq \frac{\mathcal O_X dz + \mathcal O_X d\log 0}{\mathcal O_X d\log 0}\simeq \Omega^1_{X/k}.
\]
Similarly, for an irreducible component, around a double point we have
\[
\omega_{X_v/k^\times}^1\simeq \frac{\mathcal O_{X_v} d\log x + \mathcal O_{X_v} d\log y}{\mathcal O_{X_v} (d\log x+d\log y)}.
\]
Note that only one between \(x\) and \(y\) makes sense as a coordinate on \(\mathcal O_{X_v}\), the other is just a symbol (see diagram (\ref{coordinates})). Around a smooth point of \(X^{\circ}_v\) we have \(\omega_{X_v/k^\times}^1\simeq \omega_{X/k^\times}^1\simeq \Omega^1_{X/k}\). Finally,
\[
\omega_{X_e/k^\times}^1\simeq \frac{kd\log x + k d\log y}{k(d\log x+d\log y)}.
\]

One can see that \(f_v:X_v\to k^\times\) is not log smooth, nonetheless its sheaf of relative differentials \(\omega^1_{X_v/k^\times}\) is locally free.\\
We need to introduce a log scheme \(f'_v:X'_v\to \Spec k\) with \(X'^\circ_v=X_v^\circ\) and \(\omega^1_{X_v/k^\times}\simeq \omega^1_{X'_v/k}\). It is defined as follows. Let
\[
\mathscr E_v:=\{e\in \mathscr E\mid A(e)=v\text{ or }B(e)=v\},
\]
where \(A(e)\) is the starting vertex of \(e\) and \(B(e)\) is the ending vertex of \(e\). Then \(X'_v\) is the log scheme obtained from \(X_v^\circ\) with the log structure induced by the inclusion of the divisor \(\coprod_{e\in \mathscr E_v}X_e\). Let \(e=[v,w]\), the case \(e=[w,v]\) being symmetric, a chart for \(f'_v\) around \(X_e\) is given by
\[
\begin{array}{ccc}
\mathbb N\longrightarrow \mathcal O_{X'_v}, \quad&1\longrightarrow k, \quad& 1\longrightarrow \N\nospaceperiod \\
1\longmapsto x \quad&1\longmapsto 1 \quad&1\longmapsto 0 
\end{array}
\]
Around points that are smooth in \(X^\circ\), a chart for \(f'_v\) is given by the trivial one
\[
\begin{array}{ccc}
1\longrightarrow \mathcal O_{X'_v}, \quad&1\longrightarrow k, \quad& 1\longrightarrow 1\nospaceperiod
\end{array}
\]
A local computation shows that, around a double point
\begin{align*}
\omega_{X_v/k^\times}^1\simeq \frac{\mathcal O_{X_v} d\log x + \mathcal O_{X_v} d\log y}{\mathcal O_{X_v} (d\log x+d\log y)}&\overset{\sim}{\longrightarrow}\omega^1_{X'_v/k}\simeq\mathcal O_{X'_v}d\log x\\
\left[fd\log x + gd\log y\right]&\longmapsto \left(f-g\right)d\log x.
\end{align*}
Around a smooth point \(\omega^1_{X_v/k^\times}\simeq\Omega^1_{X/k}\simeq\omega^1_{X'_v/k}\). The local isomorphisms glue to an isomorphism
\begin{equation}\label{diviso}
\omega_{X_v/k^\times}^1\simeq\omega^1_{X'_v/k}
\end{equation}

\section{Laurent Series and Residues}\label{Laurent}
Following \cite[Ch. II]{Serre88} one can associate (algebraically) Laurent series to regular functions on a smooth curve. This results in an algebraic definition of the residue of a differential form. In this section we adapt those ideas to the \'etale site.\\
Let \(k\) be an algebraically closed field of characteristic \(0\), \(X\) a smooth curve over \(k\) and \(P\) a point of \(X\). If we denote \(i:P\hookrightarrow X\) the inclusion and \(\mathcal O_P\) and \(\mathcal O_X\) the respective structure sheaves, then we have a natural map \({e}_P:\mathcal O_X \to i_*\mathcal O_P\). This morphism is the ``evaluation'' of a function on \(X\) at \(P\). If we consider an open subset \(j:U\hookrightarrow X\), we cannot expect to have a natural morphism \(j_*\mathcal O_U \to i_*\mathcal O_P\) if \(U\) does not contain \(P\). How can one evaluate a function at a point where it is not defined?\\
We will construct a (non-natural) morphism \(\varepsilon_{P}:j_*\mathcal O_U \to i_*\mathcal O_P\) extending the evaluation at \(P\). That is, \(\varepsilon_P\) fits into a commutative diagram
\[
\begin{tikzcd}
\mathcal O_X \arrow[r] \arrow[rd,"e_P"'] &j_*\mathcal O_U \arrow[d,"\varepsilon_{P}"] \\
&i_*\mathcal O_P\nospaceperiod
\end{tikzcd}
\]
The construction of \(\varepsilon_P\) \emph{depends} on the choice of a local parameter for \(\mathcal O_{X,P}\). Nonetheless, for our applications, we are interested in the morphism induced in cohomology. We will see that at this level the ambiguity disappears.\\
Elaborating on the same ideas we will come to the definition of the residue at \(P\) as a morphism of sheaves on the \'etale site
\[
\Res_P:j_*\Omega^1_{U}\longrightarrow i_*\mathcal O_P.
\]
Because it involves differentials, in contrast with the evaluation, the definition of the residue is independent on any choice.\\
Instead, for the definition of \(\varepsilon_P\), we fix a uniformizer \(x_P\) for the discrete valuation ring \(\mathcal O_{X,P}\), where the stalks are taken with respect to the Zariski topology. Let \(h:T\to X\) be an \'etale open. If \(P\not\in h(T)\), we have \(i_*\mathcal O_P(T)=0\), so we put \(\varepsilon_P(T)=0\). We now consider the case in which \(h^{-1}(P)\) is not empty. By smoothness of \(h\) and \(X\), we deduce that the local rings \(\mathcal O_{T,Q}\) are discrete valuation rings, in particular integral, for any point \(Q\) in \(T\). This implies that \(T\) is the disjoint union of integral schemes. So we can assume that \(T\) is integral. Let us now consider a point \(Q\) in \(T\) such that \(h(Q)=P\). We denote with \(\eta\) the generic point of \(T\), then the function field of \(T\) is \(\mathcal O_{T,\eta}\). The natural morphism \(\mathcal O_{T,Q}\to \mathcal O_{T,\eta}\) induces an isomorphism \(\Frac(\mathcal O_{T,Q})\simeq \mathcal O_{T,\eta}\) passing to fraction fields. We then deduce a natural map
\begin{equation}\label{mero}
\begin{tikzcd}
j_*\mathcal O_U(T)=\mathcal O(U\times_X T) \arrow[r,hook] &\Frac(\mathcal O_{T,Q})\nospaceperiod
\end{tikzcd}
\end{equation}
Since \(h\) is unramified, the natural morphism
\[
h^{\#}_{Q}:\mathcal O_{X,P}\longrightarrow \mathcal O_{T,Q}
\]
sends the uniformizer \(x_P\) to a uniformizer that we will denote \(t_Q:=h^{\#}_{Q}(x_P)\). If we denote with \(\Frac(\mathcal O_{T,Q})^{\widehat{\phantom{x}}{^Q}}\) the completion with respect to the valuation induced by \(Q\), the choice of uniformizer induces an isomorphism \(\Frac(\mathcal O_{T,Q})^{\widehat{\phantom{x}}{^Q}}\simeq k((t_Q))\), the Laurent series in \(t_Q\) with coefficients in \(k\). So, using the map in (\ref{mero}) and the inclusion in the completion, we obtain a morphism
\begin{align*}
j_*\mathcal O_U(T)&\longrightarrow \Frac(\mathcal O_{T,Q})^{\widehat{\phantom{x}}{^Q}}\simeq k((t_Q))\\
f&\longmapsto \sum\limits_{n\gg-\infty}a_{n,Q}(f)t^n_Q.
\end{align*}
Finally, for the \'etale open \(h:T\to X\), we define
\begin{align*}
\varepsilon_P(T):j_*\mathcal O_U(T)&\longrightarrow i_*\mathcal O_P(T)=\prod\limits_{Q\in h^{-1}(P)} \mathcal O_Q(Q)=\prod\limits_{Q\in h^{-1}(P)} k\\
f&\longmapsto (a_{0,Q}(f))_{Q\in h^{-1}(P)}.
\end{align*}
We are left to verify that \(\varepsilon_P\) is a morphism of sheaves. Let \(h:T\to X\) and \(h':T'\to X\) be two \'etale open with a morphism \(g:T'\to T\) over \(X\). The non-trivial case is that in which there is a point \(Q'\) in \(T'\) above \(P\). We put \(Q:=g(Q')\). As above, the choice of uniformizer \(x_P\) in \(\mathcal O_{X,P}\), induces choices of uniformizers \(t_Q\) in \(\mathcal O_{T,Q}\) and \(t'_{Q'}\) in \(\mathcal O_{T',Q'}\). These choices are compatible via \(g\). As above, we can assume that also \(T'\) is integral. We clearly get a commutative diagram
\[
\begin{tikzcd}
j_*\mathcal O_U(T)\arrow[r] \arrow[d] &\Frac(\mathcal O_{T,Q})\arrow[r] \arrow[d] &\Frac(\mathcal O_{T,Q})^{\widehat{\phantom{x}}}{^Q}\simeq k((t_Q)) \arrow[d] \\
j_*\mathcal O_U(T')\arrow[r] &\Frac(\mathcal O_{T',Q'})\arrow[r] &\Frac(\mathcal O_{T',Q'})^{\widehat{\phantom{x}}}{^{Q'}}\simeq k(({t'}_{Q'})) \nospacecomma
\end{tikzcd}
\]
where the rightmost vertical arrow sends \(t_Q\) to \(t'_{Q'}\) and is continuous, so it preserves the coefficients of a Laurent series. This shows that \(\varepsilon_P\) is a morphism of sheaves.
\begin{remark}
There is no need to assume that \(k\) is characteristic \(0\), nor to assume that it is algebraically closed. One should only replace \(k\) with the residue fields at \(P\) and \(Q\) accordingly.
\end{remark}
We will now define the residue at \(P\) morphism
\[
\Res_P:j_*\Omega^1_{U}\longrightarrow i_*\mathcal O_P.
\]
Again, for an \'etale open \(h:T\to X\), if \(P\not\in h(T)\), we have \(i_*\mathcal O_P(T)=0\) and we put \(\Res_P(T)=0\). So we assume there is at least a point \(Q\) in \(T\) such that \(h(Q)=P\). As before, we assume that \(T\) is integral and \(\eta\) will denote its generic point. We have a natural map \(\Omega^1_{T,Q}\to \Omega^1_{T,\eta}\). The choice of a uniformizer \(t_Q\) for the discrete valuation ring \(\mathcal O_{T,Q}\) gives a commutative diagram
\[
\begin{tikzcd}
\Omega^1_{T,Q}\arrow[r] \arrow[d,"\sim"]&\Omega^1_{T,\eta}\arrow[d,"\sim"] \\
\mathcal O_{T,Q} dt_Q \arrow[r] &\Frac(\mathcal O_{T,Q})dt_Q\nospaceperiod
\end{tikzcd}
\]
We then deduce a map given by
\[
\begin{tikzcd}
j_*\Omega^1_U(T)=\Omega^1(U\times_X T) \arrow[r,hook] &\Frac(\mathcal O_{T,Q})dt_Q\nospacecomma
\end{tikzcd}
\]
passing to the completion, we obtain
\begin{align*}
j_*\Omega^1_U(T)&\longrightarrow \Frac(\mathcal O_{T,Q})^{\widehat{\phantom{x}}{^Q}}dt_Q\simeq k((t_Q))dt_Q\\
\omega&\longmapsto fdt_Q=\sum\limits_{n\gg-\infty}a_{n,Q}(f)t^n_Qdt_Q.
\end{align*}
Finally, for the \'etale open \(h:T\to X\), we define
\begin{align*}
\Res_P(T):j_*\Omega^1_U(T)&\longrightarrow i_*\mathcal O_P(T)=\prod\limits_{Q\in h^{-1}(P)} \mathcal O_Q(Q)=\prod\limits_{Q\in h^{-1}(P)} k\\
\omega&\longmapsto (a_{-1,Q}(f))_{Q\in h^{-1}(P)}.
\end{align*}
In order to show that we constructed a morphism of sheaves, we will use the following result.
\begin{lemma}
In the notation above, the map \(\Res_P(T)\) does not depend on the choice of \(t_Q\).
\end{lemma}
\begin{proof}
For the proof we refer to \cite[Ch. II, §11]{Serre88}.
\end{proof}
Let \(h:T\to X\) and \(h':T'\to X\) be two \'etale open with a morphism \(g:T'\to T\) over \(X\). Again, the non-trivial case is that in which there is a point \(Q'\) in \(T'\) above \(P\), we put \(Q:=g(Q')\). We fix, as before, a uniformizer \(t_Q\) for \(\mathcal O_{T,Q}\). Then, the morphism
\[
g^{\#}_{Q'}:\mathcal O_{T,Q}\longrightarrow \mathcal O_{T',Q'}
\]
sends \(t_Q\) to a uniformizer for \(\mathcal O_{T',Q'}\) that we will denote \(t'_{Q'}:=g^{\#}_{Q'}(t_Q)\). As in the proof that the evaluation map was a morphism of sheaves, the morphism \(g\) induces a commutative diagram
\[
\begin{tikzcd}
j_*\Omega^1_U(T)\arrow[r] \arrow[d] &\Frac(\mathcal O_{T,Q})^{\widehat{\phantom{x}}{^Q}}dt_Q\simeq k((t_Q))dt_Q \arrow[d]\\
j_*\Omega^1_U(T')\arrow[r] &\Frac(\mathcal O_{T',Q'})^{\widehat{\phantom{x}}{^{Q'}}}dt'{_Q'}\simeq k((t'_{Q'}))dt'_{Q'}.
\end{tikzcd}
\]
The independence of residue on the choice of a parameter, allows us to compute \(\Res_P(T')\) with respect to \(t'_{Q'}\). From the commutativity of the diagram above, it follows that \(\Res_P\) is a morphism of sheaves.

\section{A resolution for \(\omega^\bullet_{X/k^\times}\)}\label{res}
We come back to considering a log curve \(X/k^\times\) as in §\ref{logcurve}. Our goal is to compute explicitly the cohomology of \(\omega^\bullet_{X/k^\times}\). As a first step, we will construct a Mayer-Vietoris sequence for the log-de Rham complex of \(X\). For this, we put
\[
X_0=\coprod\limits_{v \in \mathscr V} X_v\qquad\text{and}\qquad X_1=\coprod\limits_{e \in \mathscr E}X_e
\]
and denote with \(i_0:X_0\longrightarrow X\) and \(i_1:X_1\longrightarrow X\) the natural morphisms of log schemes. Given an edge \(e=[v,w]\) in the graph we can take its start \(A(e)=v\) and end \(B(e)=w\). Considering the inclusion of each point \(X_e\) in \(X_{A(e)}\) and in \(X_{B(e)}\) respectively, we obtain two morphisms of log schemes
\[
i_A,\,i_B:X_1\longrightarrow X_0.
\]
Restriction induces a morphism of complexes of sheaves
\[
\rho:\omega^\bullet_{X/k^\times}\longrightarrow i^{\circ}_{0*}\omega^\bullet_{X_0/k^\times}.
\]
Analogously we have
\[
\omega^\bullet_{X_0/k^\times}\longrightarrow i^{\circ}_{A*}\omega^\bullet _{X_1/k^\times}.
\]
After the pushforward by \(i^{\circ}_0\), this last morphism becomes
\begin{equation}\label{etaA}
\eta_A:i^{\circ}_{0*} \omega^\bullet_{X_0/k^\times} \longrightarrow i^{\circ}_{1*}\omega^\bullet _{X_1/k^\times}.
\end{equation}
Similarly, from the morphism \(i_B\), we define \(\eta_B\). We put \(\eta:=\eta_A-\eta_B\), this is the difference of restrictions.
\begin{lemma}\label{mv}
The complexes of sheaves introduced above fit into a Mayer-Vietoris exact sequence
\begin{equation}\label{sesrel}
\begin{tikzcd}
0 \arrow[r] &\omega^\bullet_{X/k^\times} \arrow[r,"\rho"] & i^{\circ}_{0*}\omega^\bullet_{X_0/k^\times} \arrow[r,"\eta"] &i^{\circ}_{1*}\omega^\bullet _{X_1/k^\times}\arrow[r]& 0\nospaceperiod
\end{tikzcd}
\end{equation}
\end{lemma}
\begin{proof}
\'Etale locally around a smooth point, \(X^{\circ}\) is an \'etale pullback of \(\mathbb A_k^1\) and there is nothing to prove. \'Etale locally around a double point, \(X^{\circ}\) is an \'etale open of \(\Spec[x,y]/(xy)\). So we are left to prove the statement for \(X=\Spec[x,y]/(xy)\) and to consider global sections, indeed sections on an \'etale open will be obtained by tensoring the exact sequence by a flat module. The proof is now reduced to a direct calculation.\\
\end{proof}
We introduce the double complex
\[
\mathcal A_{X/k^\times}=
\begin{tikzcd}
\bigoplus\limits_{v\in\mathscr V}i^{\circ}_{v*}\omega^1_{X_v/k^\times} \arrow[r,"\eta^1"] &\bigoplus\limits_{e\in\mathscr E}i^{\circ}_{e*}\omega^1_{X_e/k^\times}\\
\bigoplus\limits_{v\in\mathscr V}i^{\circ}_{v*}\mathcal O_{X_v} \arrow[r,"\eta^0"]\arrow[u,"d"] &\bigoplus\limits_{e\in\mathscr E}i^{\circ}_{e*}\mathcal O_{X_e}\arrow[u,"d"]\nospaceperiod
\end{tikzcd}
\]
The lemma above shows that we have a quasi-isomorphism in the category of complexes of sheaves
\begin{equation}\label{qisotot}
\rho:\omega^\bullet_{X/k^\times}\longrightarrow\Tot(\mathcal A_{X/k^\times}).
\end{equation}
For each vertex \(v\) in \(\mathscr V\), we will call \(U_v=X^{\circ}_v\setminus \coprod_{e \in \mathscr E_v} X^{\circ}_e\), this is an affine smooth curve over \(k\). We denote \(j_v:U_v\hookrightarrow X^{\circ}\) and \(j'_v:U_v\hookrightarrow X^{\circ}_v\) the natural open embeddings. For each point of intersection \(i'^{\circ}_e:X^{\circ}_e\hookrightarrow X^{\circ}_v\), we fix a uniformizer \(x_e\) for \(\mathcal O_{X_v,X_e}\). Following section \ref{Laurent}, we have a morphism of sheaves on \(X^{\circ}_v\) ``evaluation at \(X_e\)''
\[
\varepsilon_{X_e}:j'_{v*}\mathcal O_{U_v}\longrightarrow i'^{\circ}_{e*}\mathcal O_{X_e},
\]
after pushforward via \(i^{\circ}_v:X^{\circ}_v\hookrightarrow X^{\circ}\), we obtain a morphism on \(X^{\circ}\) that we will denote as follows
\[
(-)|_{X_e}:j_{v*}\mathcal O_{U_v}\longrightarrow i^{\circ}_{e*}\mathcal O_{X_e}.
\]
Analogously, the residue \(\Res_{X_e}:j'_{v*}\Omega^1_{U_v}\longrightarrow i'_{e*}\mathcal O_{X_e}\) induces, after pushforward via \(i_v\), a morphism that for simplicity we keep calling in the same way:
\[
\Res_{X_e}:j_{v*}\Omega^1_{U_v}\longrightarrow i^{\circ}_{e*}\mathcal O_{X_e}.
\]
Using the maps we have introduced we construct a double complex
\[
\mathcal B_{X/k^\times}=
\begin{tikzcd}
\bigoplus\limits_{v\in\mathscr V}j_{v*}\Omega_{U_v}^1 \arrow[r,"\zeta^1"] &\bigoplus\limits_{e\in\mathscr E}i^{\circ}_{e*}\omega^1_{X_e/k^\times}\\
\bigoplus\limits_{v\in\mathscr V}j_{v*}\mathcal O_{U_v} \arrow[r,"\zeta^0"]\arrow[u,"d"] &\bigoplus\limits_{e\in\mathscr E}i^{\circ}_{e*}\mathcal O_{X_e}\arrow[u,"d=0"]\nospaceperiod
\end{tikzcd}
\]
The horizontal arrows are given by
\begin{align*}
\zeta^1(\omega_v)_v&=\left(\left[\Res_{X_e}(\omega_v)d\log x - \Res_{X_e}(\omega_w)d\log y\right]\right)_{e=[v,w]},\\
\zeta^0(f_v)_v&=\left(f_v|_{X_e}-f_w|_{X_e}\right)_{e=[v,w]}.
\end{align*}
\begin{remark}
As pointed out in §\ref{Laurent}, our definition of \(f_v|_{X_e}\) relies on the choice of a uniformizer. Nonetheless, the cohomology groups \(\HHH^i(X,\Tot(\mathcal B_{X/k^\times}))\) are unaffected by this decision. Indeed, Lemma \ref{expl} will show that \(\HHH^i(X,\Tot(\mathcal B_{X/k^\times}))\) is naturally isomorphic to \(\HH^i_{\log}(X/k^{\times})\). Clearly, these last groups are independent on the choice above.
\end{remark}
Now, we take global sections indexwise and obtain
\[
B_{X/k^\times}=
\begin{tikzcd}
\bigoplus\limits_{v\in\mathscr V}\HH^0(U_v,\Omega_{U_v}^1) \arrow[r,"\zeta^1"] &\bigoplus\limits_{e\in\mathscr E}\HH^0(X_e,\omega^1_{X_e/k^\times})\\
\bigoplus\limits_{v\in\mathscr V}\HH^0(U_v,\mathcal O_{U_v}) \arrow[r,"\zeta^0"]\arrow[u,"d"] &\bigoplus\limits_{e\in\mathscr E}\HH^0(X_e,\mathcal O_{X_e})\arrow[u,"d=0"]\nospaceperiod
\end{tikzcd}
\]
Then the cohomology of this complex computes \(\HH^i_{\log}(X/k^{\times})\).
\begin{lemma}\label{expl} 
With the notation above, we have an isomorphism
\[
\HH^i_{\log}(X/k^\times)\simeq \HH^i(\Tot(B_{X/k^\times})),
\]
for any integer \(i\).
\end{lemma}
\begin{proof}
We consider the morphism of complexes
\begin{equation*}
\beta:i_{v*}^\circ\omega_{X'_v/k}^\bullet\longrightarrow j_{v*}\Omega_{U_v}^\bullet.
\end{equation*}
We compose it with the isomorphism in (\ref{diviso}) and obtain a morphism
\[
\beta':i_{v*}^\circ\omega_{X_v/k^{\times}}^\bullet\longrightarrow j_{v*}\Omega_{U_v}^\bullet.
\]
Applying \(\beta'\) and the identity to the first and second column respectively, we get a morphism of double complexes \(\alpha:\mathcal A_{X/k^\times}\to\mathcal B_{X/k^\times}\). The quasi-isomorphism in (\ref{qisotot}) tells us that \(\HH^i_{\log}(X/k^\times)\simeq\HHH^i(X,\Tot(\mathcal A_{X/k^\times}))\). The fact that the \(U_v\) and the \(X_e\) are affine implies
\[
\HHH^i(X,\Tot(\mathcal B_{X/k^\times}))\simeq \HH^i(\Tot(B_{X/k^\times})).
\]
So we are reduced to show that \(a:\HHH^i(X,\Tot(\mathcal A_{X/k^\times}))\to\HHH^i(X,\Tot(\mathcal B_{X/k^\times}))\), the morphism induced by \(\alpha\) in hypercohomology, is an isomorphism. To show this, we take the spectral sequences (cf. \cite[Ch. 0, §11]{EGAIII_I})
\[
E_{\mathcal A,1}^{p,q}=\HHH^q(X,\mathcal A_{X/k^\times}^{p,\bullet})\Longrightarrow \HHH^{p+q}(X,\Tot(\mathcal A_{X/k^\times}))
\]
and
\[
E_{\mathcal B,1}^{p,q}=\HHH^q(X,\mathcal B_{X/k^\times}^{p,\bullet})\Longrightarrow \HHH^{p+q}(X,\Tot(\mathcal B_{X/k^\times})).
\]
The morphism \(\alpha\) induces a morphism of spectral sequences \(a'_{p,q}:E^{p,q}_{\mathcal A,1}\to E^{p,q}_{\mathcal B,1}\). We want to show that \(a'\) is an isomorphism. For \(p=1\), \(a'_{1,q}\) is the identity. The two spectral sequences have only two non-zero colomuns, hence we are left to show that the morphism on the first column \(a'_{0,q}\) is an isomorphism. Since \(a'_{0,q}\) is obtained from \(\beta'\) taking hypercohomology, using (\ref{diviso}), we can look at
\[
b:\HHH^i(X,i_{v*}^\circ\omega_{X'_v/k}^\bullet)\longrightarrow \HHH^i(X,j_{v*}\Omega_{U_v}^\bullet).
\]
Now, similarly to \cite[Prop. II 3.13]{Deligne70}, one shows that \(b\) is an isomorphism. To conclude, we have seen that \(a'\) is an isomorphism, so we deduce that the morphism between the abutments \(a:\HHH^i(X,\Tot(\mathcal A_{X/k^\times}))\to\HHH^i(X,\Tot(\mathcal B_{X/k^\times}))\) is an isomorphism.
\end{proof}

\section{Combinatorial Monodromy, Du Bois Cohomology and Invariant Cycles}\label{mon}
We construct an endomorphism \(\tilde{N}\) on the cohomology of a log curve \(\HH^1_{\log}(X/k^\times)\), notation as in §\ref{logcurve}. We will refer to this operator as ``combinatorial monodromy''. We have two justifications for the use of the term ``monodromy''. First, as we will see later in §\ref{comp}, when \(X\) comes from a simple normal crossing divisor in a proper family over the complex disc, this operator coincides with the classical monodromy (Theorem \ref{moncomp}). Second, in this section we show that \(\tilde{N}\) provides an invariant cycles exact sequence (Theorem \ref{thm2}) in analogy with other monodromies, as for the complex and \(\ell\)-adic situation (see \cite[§1.7]{DeCataldo09} and \cite[Corollaire 6.2.8]{Beilinson82}) and for the \(p\)-adic Hyodo-Kato monodromy (\cite{Chiarellotto99} and \cite{Chiarellotto16}).\\
Let \([\omega]\) in \(\HH^1_{\log}(X/k^\times)\) be a cohomology class. Using the description of Lemma \ref{expl}, the element \([\omega]\) can be represented by a hypercocycle \(((\omega_v)_{v\in \mathscr V},(f_e)_{e \in \mathscr E})\) with \(\omega_v\) in \(\HH^0(U_v,\Omega_{U_v/k}^1)\) and \(f_e\) in \(\HH^0(X_e,\mathcal O_{X_e})=k\) such that
\[
\Res_{X_e}(\omega_v)+\Res_{X_e}(\omega_w)=0,
\]
for every \(e=[v,w]\). We define the combinatorial monodromy as follows
\begin{align*}\label{combmon}
\tilde{N}:\HH^1_{\log}(X/k^\times)&\longrightarrow\HH^1_{\log}(X/k^\times)\\
[\omega]=\left[\left(\left(\omega_v\right)_{v\in \mathscr V},\left(f_e\right)_{e \in \mathscr E}\right)\right]&\longmapsto\left[\left(0,\left(\Res_{X_e}\left(\omega_v\right)\right)_{e=[v,w]}\right)\right].
\end{align*}
To have invariant cycles we want to give a cohomological descripition for \(\Ker\tilde{N}\) in terms of \(X\). In the topological setting (see §\ref{intro} and §\ref{comp}) this should be given by the singular cohomology of \(X^{\circ}\). This group can also be computed using a complex of differentials, the Du Bois complex (cf. \cite[Th\'eor\`eme 4.5]{DuBois81}). Since its definition makes sense algebraically, we adapt it to our situation and put
\[
\underline{\Omega}_{X/k}^\bullet=\Ker\left(\bigoplus\limits_{v\in \mathscr V}i_{v*}^\circ \Omega_{X_v/k}^\bullet \overset{\underline{\eta}}{\longrightarrow} \bigoplus\limits_{e\in \mathscr E}i_{e*}^\circ \Omega_{X_e/k}^\bullet\right).
\]
The morphism \(\underline{\eta}\) is the ordered difference as in (\ref{sesrel}). This is exactly the descripition in \cite[p. 69]{DuBois81} of the Du Bois complex for normal crossing divisors. We observe that we have a natural morphism
\[
\underline{\Omega}_{X/k}^\bullet\longrightarrow\bigoplus\limits_{v\in\mathscr V}i^{\circ}_{v*}\omega^\bullet_{X_v/k^\times}\longrightarrow\Tot(\mathcal A_{X/k^\times}),
\]
composing it the quasi-inverse of \(\rho\) (see (\ref{qisotot})) we obtain a morphism in the derived category
\[
\underline{sp}:\underline{\Omega}_{X/k}^\bullet\to \omega^\bullet_{X/k^\times}.
\]
We define the Du Bois cohomology of \(X\), for any integer \(i\), as \(\HH^i_{DB}(X/k)=\HHH^i(X,\underline{\Omega}_{X/k}^\bullet)\) and we call \(\underline{sp}\) also the morphism induced in hypercohomology.
\begin{theorem}\label{thm2}
We have an exact sequence
\[
0\longrightarrow\HH^1_{DB}(X/k)\overset{\underline{sp}}{\longrightarrow}\HH^1_{\log}(X/k^{\times})\overset{\tilde{N}}{\longrightarrow}\HH^1_{\log}(X/k^{\times}).
\]
\end{theorem}
\begin{proof}
The Du Bois complex \(\underline{\Omega}_{X/k}^\bullet\) is quasi-isomorphic to the total complex of
\[
\underline{\mathcal A}_{X/k}=
\begin{tikzcd}
\bigoplus\limits_{v\in\mathscr V}i^{\circ}_{v*}\Omega^1_{X_v/k} \arrow[r,"\underline{\eta}^1"] &0\\
\bigoplus\limits_{v\in\mathscr V}i^{\circ}_{v*}\mathcal O_{X_v} \arrow[r,"\underline{\eta}^0"]\arrow[u,"d"] &\bigoplus\limits_{e\in\mathscr E}i^{\circ}_{e*}\mathcal O_{X_e}\arrow[u,"d"]\nospaceperiod
\end{tikzcd}
\]
We have a morphism of double complexes \(\underline{\mathcal A}_{X/k}\to\mathcal A_{X/k^\times}\). As in the proof of Lemma \ref{expl}, we have two spectral sequences
\[
E_{\underline{\mathcal A},1}^{p,q}=\begin{tikzcd}
\bigoplus\limits_{v\in\mathscr V}\HH^1_{dR}(X_v/k) \arrow[r] &0 \\
\bigoplus\limits_{v\in\mathscr V}\HH^0_{dR}(X_v/k) \arrow[r,"\underline{\alpha}"] &\bigoplus\limits_{e\in\mathscr E}\HH^0_{dR}(X_e/k)
\end{tikzcd}\Longrightarrow\HH^{p+q}_{DB}(X/k)
\]
and
\[
E_{\mathcal A,1}^{p,q}=\begin{tikzcd}
\bigoplus\limits_{v\in\mathscr V}\HH^1_{\log}(X_v/k^\times) \arrow[r,"\beta"] &\bigoplus\limits_{e\in\mathscr E}\HH^1_{\log}(X_e/k^\times) \\
\bigoplus\limits_{v\in\mathscr V}\HH^0_{\log}(X_v/k^\times) \arrow[r,"\alpha"] &\bigoplus\limits_{e\in\mathscr E}\HH^0_{\log}(X_e/k^\times)
\end{tikzcd}\Longrightarrow\HH^{p+q}_{\log}(X/k^\times).
\]
The associated filtrations on the abutments give us a commutative diagram with exact rows
\begin{equation}\label{snakediag}
\begin{tikzcd}
&0 \arrow[d] &0 \arrow[d] &0 \arrow[d] \\
0 \arrow[r] &\Coker\underline{\alpha} \arrow[d,"\sim"] \arrow[r] &\HH^1_{DB}(X/k) \ar[d,"\underline{sp}"] \arrow[r] &\bigoplus\limits_{v \in \mathscr V}\HH^1_{dR}(X_v/k) \arrow[d,"\varphi"] \arrow[r] &0 \\
0 \arrow[r] &\Coker\alpha  \arrow[d] \arrow[r] &\HH^1_{\log}(X/k^\times)\arrow[d,"\tilde{N}"] \arrow[r] &\Ker\beta \arrow[d,"N'"] \arrow[r] &0\\
0 \arrow[r] &0 \arrow[r] &\HH^1_{\log}(X/k^\times) \arrow[r,"\id"] &\HH^1_{\log}(X/k^{\times}) \arrow[r] &0.
\end{tikzcd}
\end{equation}
Here \(N'\) is the morphism induced by \(\tilde{N}\) on \(\Ker \beta\). Since the first column is clearly exact, we are reduced to show that the rightmost column is exact. This is proved in the following lemma.
\end{proof}

\begin{lemma}\label{redinvcyc}
The sequence
\[
0\longrightarrow\bigoplus\limits_{v \in \mathscr V} \HH^1_{dR}(X_v/k) \overset{\varphi}{\longrightarrow}\Ker\beta\overset{N'}{\longrightarrow}\HH^1_{\log}(X/k^{\times}),
\]
in diagram (\ref{snakediag}) is exact.
\end{lemma}
\begin{proof}
We have showed that \(\HH^1_{\log}(X_v/k^\times)\simeq\HH^1_{\log}(X'_v/k)\simeq\HH^1_{dR}(U_v/k)\), so a class \([\omega_v]\) in \(\HH^1_{\log}(X_v/k^\times)\) will be represented by a section \(\omega_v\) in \(\HH^0(U_v,\Omega^1_{U_v/k})\). Then \(([\omega_v])_v\) is in \(\Ker\beta\) if and only if
\begin{equation}\label{sumreszero}
\Res_{X_e}(\omega_v)+\Res_{X_e}(\omega_w)=0,
\end{equation}
for any \(e=[v,w]\). We have
\begin{align*}
\Ker N'&=\{([\omega_v])_v \in \Ker\beta\mid[(0,\Res_{X_e}(\omega_v)_{e=[v,w]})]=0\text{ in }\HH^1_{\log}(X/k^\times)\}.
\end{align*}
We take a class \(([\omega_v])_v\) in \(\Ker N'\), the residue theorem (\cite[Ch.II, Prop. 6]{Serre88}) tells us that, for every vertex \(v\) we have
\begin{equation}\label{resthm}
\sum\limits_{e \in \mathscr E_v}\Res_{X_e}(\omega_v)=0.
\end{equation}
For each vertex \(v\), we put
\[
\mathscr E_v^{+}=\{e\in\mathscr E|A(e)=v\}\qquad\text{and}\qquad\mathscr E_v^{-}=\{e\in\mathscr E|B(e)=v\}.
\]
Given an edge \(e=[v,w]\), we define \(a_e=\Res_{X_e}(\omega_v)\), so we rewrite (\ref{resthm}) as
\begin{align}\label{condcomblemma}
\sum\limits_{e \in \mathscr E_v}\Res_{X_e}(\omega_v)&=\sum\limits_{e \in \mathscr E_v^+}\Res_{X_e}(\omega_v)+\sum\limits_{e \in \mathscr E_v^-}\Res_{X_e}(\omega_v)\\
&\overset{\mathclap{\text{(\ref{sumreszero})}}}{=}\sum\limits_{e \in \mathscr E_v^+}a_e-\sum\limits_{e \in \mathscr E_v^-}a_e=0 \nonumber.
\end{align}
On the other hand \(([\omega_v])_v\) is in \(\Ker N'\), so \([(0,a_e)]=0\) in \(\HH^1_{\log}(X/k^\times)\). This happens if and only if there exist constants \((a_v)_v\) in \(\oplus_{v\in \mathscr V} k\) such that \(a_v-a_w=a_e\), for \(e=[v,w]\). Since condition (\ref{condcomblemma}) holds, a combinatorial argument, used also in the \(p\)-adic setting (\cite[Proposition 8]{Chiarellotto16}, Lemma \ref{comblemma} below), shows that \(a_e=0\) for every \(e\). This proves
\begin{align*}
\Ker N'&=\{([\omega_v])_v \in \Ker\beta\mid\Res_{X_e}(\omega_v)=0\text{ for }e=[v,w]\}.
\end{align*}
Now, for any vertex \(v\) we have a Gysin exact sequence (\cite[III, §8.3]{Hartshorne70}) for the embeddings \(\coprod_{e\in \mathscr E_v} X_e \hookrightarrow X_v \hookleftarrow U_v\) as below:
\[
0 \longrightarrow \HH^1_{dR}(X_v/k)\overset{\varphi_v}{\longrightarrow}\HH^1_{dR}(U_v/k)\overset{\Res}{\longrightarrow}\bigoplus\limits_{e\in \mathscr E_v}\HH^0_{dR}(X_e/k)\longrightarrow \cdots.
\]
This proves that \((\varphi_v)_v\) factors through \(\varphi\), so we deduce injectivity of \(\varphi\) and that \(\IIm \varphi = \Ker N'\).
\end{proof}
\begin{lemma}\cite[Proposition 8]{Chiarellotto16}\label{comblemma}
Let \(\mathbb F\) be a field of characteristic \(0\) and \(G=(\mathscr V,\mathscr E)\) be a finite directed graph, connected and without loops. Define the linear morphism of vector spaces
\begin{align*}
\varphi:\oplus_{v\in \mathscr V}\mathbb F &\longrightarrow \oplus_{e\in \mathscr E}\mathbb F\\
(a_v)_v &\longmapsto (a_v-a_w)_{e=[v,w]}.
\end{align*}
Then
\[
\IIm(\varphi)\cap\{(a_e)_e\in \oplus_{e\in \mathscr E}\mathbb F  \mid \sum_{e \in \mathscr E_v^+}a_e-\sum_{e \in \mathscr E_v^-}a_e=0\}=0.
\]
\end{lemma}

\section{Monodromy and Gauss-Manin connection}\label{mongm}
In this section we consider the same setting of the introduction. That is, \(S/\C\) is a smooth connected affine curve, \(\pi:\mathfrak X\to S\) a proper, separated morphism of finite type from a smooth scheme \(\mathfrak X/\C\). We assume that the fibers \(\mathfrak X_s\) are smooth curves when \(s\in S^*=S\setminus \{P\}\) and that \(\mathfrak X_P\) is a simple normal crossing divisor. By Ehresmann's Lemma, the restriction \(\pi':\mathfrak X'\to (S^*)^{an}\) is a locally trivial \(C^{\infty}\)-fibration (here \(\mathfrak X'\) denotes \((\mathfrak X\setminus \mathfrak X_P)^{an}\)). We take an open punctured disc \(\Delta^*\subset (S^*)^{an}\) around \(P\) and fix a point \(s\in \Delta^*\). The positive generator of \(\pi_1(\Delta^*,s)\simeq \Z\) acts on \(\HH^1(\mathfrak X_s^{an};\C)\) via a monodromy endomorphism that we will denote \(T_s\). This, a priori topological, operator has an algebraic description via the Gauss-Manin connection. We will review this construction using the language of log schemes.\\
We look at \(\mathfrak X\) with the divisorial log structure \(\mathcal M_{\mathfrak X}\) induced by \(\mathfrak X_P\), as defined in \cite[§1.5]{Kato89}, and denote by \(\mathfrak X^{\times}=(\mathfrak X,\mathcal M_{\mathfrak X})\) the log scheme obtained. Following the same convention, \(S^{\times}\) will be the log scheme with the divisorial log structure induced by \(P\). Then the morphism \(\pi\) underlies a log smooth morphism \(\pi^{\times}:\mathfrak X^{\times}\to S^{\times}\). Pulling back the log structures on \(\mathfrak X\) and \(S\), respectively to \(\mathfrak X_P\) and \(P\) we obtain log schemes \(\mathfrak X_P^{\times}\) and \(P^{\times}\). Then \(\mathfrak X_P^{\times}/P^{\times}\), is a log curve as in the hypotheses of §\ref{logcurve}, so we put \(\mathfrak X_P^{\times}=X\) and \(P^{\times}=k^{\times}\), with \(k=\C\). The setup is summarized by the following cartesian diagram of log schemes
\[
\begin{tikzcd}
X\arrow[r,"i"] \arrow[d,"f"'] &\mathfrak X^{\times} \arrow[d,"\pi^{\times}"]\\
k^{\times}\arrow[r,"i_P"]&S^{\times}\nospacecomma
\end{tikzcd}
\]
where the vertical arrows are log smooth. We have sheaves of log-de Rham differentials \(\omega^1_{{\mathfrak X}^{\times}/S^{\times}}\) and \(\omega^1_{X/k^{\times}}\simeq i^*(\omega^1_{\mathfrak X^{\times}/S^{\times}})\) that are locally free by log smoothness. Taking wedge products, one obtains the log-de Rham complexes \(\omega^{\bullet}_{{\mathfrak X}^{\times}/S^{\times}}\) and \(\omega^{\bullet}_{X/k^{\times}}\).
In contrast with the rest of the paper, in this section we consider log schemes (and consequently sheaves of log differentials) with respect to the Zariski topology. More precisely, if \((Z_{\acute et},\mathcal M_Z)\) is any of the log schemes introduced above and \(\varepsilon:Z_{\acute et}\to Z_{Zar}\) is the usual morphism of sites, we will look at the log schemes
\((Z_{Zar},\varepsilon _* \mathcal M_Z)\).
\begin{remark}\label{zaret}
Alternatively, one can start with the log scheme \((Z_{Zar},\mathcal M_{Z_{Zar}})\), where the divisorial log structure (resp. the pullback of the divisorial log structure) is defined on the Zariski site. This yields the same log scheme as the one taken above: \((Z_{Zar},\mathcal M_{Z_{Zar}})\simeq (Z_{Zar},\varepsilon _* \mathcal M_Z)\). Moreover, if \({\omega}^1_{Z_{Zar}}\) and \({\omega}^1_{Z_{\acute et}}\) denote the sheaves of differentials of the respective log schemes, then \({\omega}^1_{Z_{Zar}}\simeq \varepsilon_*{\omega}^1_{Z_{\acute et}}\).
\end{remark}
For the rest of the section, We will omit the subscript ``\(Zar\)'' and the pushforward ``\(\varepsilon_*\)'' to simplify the exposition. Without further assumptions, all sheaves are intended on the Zariski site.
\begin{remark} As in \cite[§0]{Illusie94a}, we stress the fact that, after analytification, the complex \(\omega^{\bullet}_{X/k^{\times}}\) is the complex of sheaves studied in \cite{Steenbrink76} to compute the limit Hodge structure.
\end{remark}
The absolute log-de Rham complex for \(\mathfrak X^{\times}\) fits into a short exact sequences of complexes of sheaves (cf. \cite[Proposition 3.12]{Kato89})
\begin{equation}\label{sesfam}
0\longrightarrow\pi^* \omega^1_{S^{\times}}\otimes_{\mathcal O_{\mathfrak X}}\omega^{\bullet-1}_{{\mathfrak X}^{\times}/S^{\times}}\overset{\wedge}{\longrightarrow}\omega^{\bullet}_{{\mathfrak X}^{\times}}\longrightarrow\omega^{\bullet}_{{\mathfrak X}^{\times}/S^{\times}}\longrightarrow 0,
\end{equation}
where the first non-zero arrow is \(\eta\otimes \eta'\mapsto \eta \wedge \eta'\).\\ 
The exact sequence induces a morphism in \(D^b(\mathfrak X)\), the bounded derived category of complexes of abelian sheaves,
\[
\omega^{\bullet}_{{\mathfrak X}^{\times}/S^{\times}}\longrightarrow\pi^*\omega^1_{S^{\times}}\otimes_{\mathcal O_{\mathfrak X}}\omega^{\bullet-1}_{{\mathfrak X}^{\times}/S^{\times}}[1].
\]
After changing signs to make the differentials agree (cf. §\ref{conv}), i.e. composing with  \((-1)^i{\id}:\pi^*\omega^1_{S^{\times}}\otimes_{\mathcal O_{\mathfrak X}}\omega^{\bullet-1}_{{\mathfrak X}^{\times}/S^{\times}}[1]\to\pi^*\omega^1_{S^{\times}}\otimes_{\mathcal O_{\mathfrak X}}\omega^{\bullet}_{{\mathfrak X}^{\times}/S^{\times}}\), we get
\begin{equation}\label{connmorph}
\delta:\omega^{\bullet}_{{\mathfrak X}^{\times}/S^{\times}}\longrightarrow\pi^*\omega^1_{S^{\times}}\otimes_{\mathcal O_{\mathfrak X}}\omega^{\bullet}_{{\mathfrak X}^{\times}/S^{\times}}.
\end{equation}
We apply the derived pushforward \(\mathbb R^i\pi_*\) and obtain the Gauss-Manin connection
\[
\nabla:\mathbb R^i\pi_*\omega^{\bullet}_{{\mathfrak X}^{\times}/S^{\times}}\longrightarrow\mathbb R^i\pi_*(\pi^*\omega^1_{S^{\times}}\otimes_{\mathcal O_{\mathfrak X}}\omega^{\bullet}_{{\mathfrak X}^{\times}/S^{\times}})\simeq\omega^1_{S^{\times}}\otimes_{\mathcal O_S}\mathbb R^i\pi_*\omega^{\bullet}_{{\mathfrak X}^{\times}/S^{\times}},
\]
where the isomorphism is the projection formula for complexes of \(\pi^{-1}\mathcal O_S\)-modules (cf. \cite[5.4.10]{EGAI}). We have a residue at \(P\) morphism
\begin{align*}
\Res_{P}:\omega^1_{S^{\times}}&\longrightarrow \mathcal O_S\\
f\frac{dt}{t}&\longmapsto f,
\end{align*}
the residue of the Gauss-Manin connection at \(P\) is obtained from the composition
\[
\mathbb R^i\pi_*\omega^{\bullet}_{{\mathfrak X}^{\times}/S^{\times}}\overset{\nabla}{\longrightarrow}\omega^1_{S^{\times}}\otimes_{\mathcal O_S}\mathbb R^i\pi_*\omega^{\bullet}_{{\mathfrak X}^{\times}/S^{\times}}\overset{\Res_P\otimes \id}{\longrightarrow} \mathcal O_{S}\otimes_{\mathcal O_S}\mathbb R^i\pi_*\omega^{\bullet}_{{\mathfrak X}^{\times}/S^{\times}}
\]
and taking the fiber
\[
\Res_P\nabla:i_{P}^*\R^i \pi_*\omega^{\bullet}_{{\mathfrak X}^{\times}/S^{\times}}\longrightarrow i_{P}^*\R^i \pi_*\omega^{\bullet}_{{\mathfrak X}^{\times}/S^{\times}}.
\]
\begin{remark}\label{derivedgm}
If one considers the composition of morphisms in the derived category
\[
\omega^{\bullet}_{{\mathfrak X}^{\times}/S^{\times}}\overset{\delta}{\longrightarrow}\pi^*\omega^1_{S^{\times}}\otimes_{\mathcal O_{\mathfrak X}}\omega^{\bullet}_{{\mathfrak X}^{\times}/S^{\times}}\overset{\pi^*\Res_P\otimes \id}{\longrightarrow}\pi^*\mathcal O_{S}\otimes_{\mathcal O_{\mathfrak X}}\omega^{\bullet}_{{\mathfrak X}^{\times}/S^{\times}},
\]
then \(\Res_P(\nabla)=i_P^*\R^i\pi_*((\pi^*\Res_P\otimes \id)\circ \delta)\). Indeed, using the projection formula, the residue commutes with the derived pushforward.
\end{remark}
We collect here the results showing that the residue of the Gauss-Manin connection is a fiber of an automorphism of \(\R^i \pi_*\omega^{\bullet}_{{\mathfrak X}^{\times}/S^{\times}}\), the others being the monodromy operators on the smooth fibers. We first present an algebraic version of \cite[Theorem 2.18]{Steenbrink76}, one can prove it as in the original account taking analytifications and comparing cohomologies.
\begin{proposition}\label{locfreesten}
For all integers \(i\), the sheaf \(\R^i \pi_*(\omega^{\bullet}_{{\mathfrak X}^{\times}/S^{\times}})\) is locally free and the natural map
\begin{equation}\label{Steenbrinkiso}
i_{P}^*\R^i \pi_*(\omega^{\bullet}_{{\mathfrak X}^{\times}/S^{\times}})\longrightarrow \HHH^i(X,\omega^{\bullet}_{X/k^{\times}})
\end{equation}
is an isomorphism.
\end{proposition}
\begin{proof}
For any point \(s\) in \(S\), let \(i_s:\mathfrak X_s\hookrightarrow \mathfrak X\) denote the embedding of the respective fiber. For any integer \(i\) we look at the function
\begin{equation}\label{constdim}
s\longmapsto \dim_{\C}\HHH^i(\mathfrak X_s, i_s^{*}\omega^{\bullet}_{\mathfrak X^{\times}/S^{\times}}).
\end{equation}
Via GAGA correspondence we have
\[
\HHH^i(\mathfrak X_s, i_s^{*}\omega^{\bullet}_{\mathfrak X^{\times}/S^{\times}})\simeq\HHH^i(\mathfrak X^{an}_s, (i_s^{*}\omega^{\bullet}_{\mathfrak X^{\times}/S^{\times}})^{an}).
\]
So using \cite[propositions 2.2 and 2.16]{Steenbrink76}, we deduce that the function in (\ref{constdim}) is constant. We can then apply \cite[Corollary 2, p. 50]{Mumford08} and conclude.
\end{proof}
We define
\[
N=\theta\circ\Res_P\nabla\circ\theta^{-1}
\]
where \(\theta\) is the isomorphims in (\ref{Steenbrinkiso}). We take the short exact sequence for relative log differentials on \(X\) analogue of (\ref{sesfam}), obtained from this by pullback,
\begin{equation}\label{sesfib}
0\longrightarrow f^* \omega^1_{k^{\times}}\otimes_{\mathcal O_{X}}\omega^{\bullet-1}_{X/k^{\times}}\overset{\wedge}{\longrightarrow}\omega^{\bullet}_{X}\longrightarrow\omega^{\bullet}_{X/k^{\times}}\longrightarrow 0
\end{equation}
and so, as for (\ref{connmorph}), we obtain a connecting morphism in \(D^b(X)\), the derived category of complexes of abelian sheaves on \(X\),
\begin{equation}\label{connmorphpartial}
\partial:\omega^{\bullet}_{X/k^{\times}}\longrightarrow f^*\omega^1_{k^{\times}}\otimes_{\mathcal O_{X}}\omega^{\bullet}_{X/k^{\times}}\simeq \omega^{\bullet}_{X/k^{\times}}.
\end{equation}
We show that this connecting morphism computes \(N\) (see also \cite[proof of Proposition 2.20]{Steenbrink76} and \cite[Proposition 11.15]{Peters08}).
\begin{proposition}\label{stenprop}
The endomorphism \(N\) on \(\HHH^1(X,\omega^{\bullet}_{X/k^{\times}})\) is the map induced by \(\partial\) in hypercohomology.
\end{proposition}
\begin{proof}
Let us consider the composition
\[
\omega^{\bullet}_{{\mathfrak X}^{\times}/S^{\times}}\overset{\delta}{\longrightarrow}\pi^*\omega^1_{S^{\times}}\otimes_{\mathcal O_{\mathfrak X}}\omega^{\bullet}_{{\mathfrak X}^{\times}/S^{\times}}\overset{\pi^*\Res_P\otimes \id}{\longrightarrow}\pi^*\mathcal O_{S}\otimes_{\mathcal O_{\mathfrak X}}\omega^{\bullet}_{{\mathfrak X}^{\times}/S^{\times}},
\]
its index-wise pullback via \(i:X\hookrightarrow \mathfrak X\) is a morphism in the derived category, indeed it is \(\partial:\omega^{\bullet}_{X/k^{\times}}\to  \omega^{\bullet}_{X/k^{\times}}\). So the morphism induced in hypercohomology by \(\partial\) is \(\R^1f_*(\partial)=\R^1f_*i^*((\pi^*\Res_P\otimes \id)\circ \delta)\). We recall that we have a natural transformation \(i_P^*\R^1\pi_*\to\R^1f_*\) and so a commutative diagram
\[
\begin{tikzcd}
i_P^*\R^1\pi_*\omega^{\bullet}_{{\mathfrak X}^{\times}/S^{\times}}\arrow[r,"i_P^*\R^i\pi_*((\pi^*\Res_P\otimes \id)\circ \delta)"] \arrow[d,"\theta"'] &[8em]i_P^*\R^1\pi_*\omega^{\bullet}_{{\mathfrak X}^{\times}/S^{\times}}\arrow[d,"\theta"]\\
\R^1f_*i^*\omega^{\bullet}_{{\mathfrak X}^{\times}/S^{\times}}\arrow[r,"\R^1f_*i^*((\pi^*\Res_P\otimes \id)\circ \delta)"] &\R^1f_*i^*\omega^{\bullet}_{{\mathfrak X}^{\times}/S^{\times}}\nospaceperiod
\end{tikzcd}
\]
Proposition \ref{locfreesten} shows that the vertical arrows are isomorphisms, so the horizontal arrows have to agree, by Remark \ref{derivedgm} we conclude.
\end{proof}
\begin{remark}\label{ev0}
As in \cite[Proposition 2.20]{Steenbrink76}, a local calculation (for the \'etale topology or analytically via GAGA), shows that \(\Res_P \nabla\) has only \(0\) as eigenvalue (we are in the reduced case). Alternatively, one can deduce the nilpotency of \(\Res_P \nabla\) from Theorem \ref{moncomp} that is independent from this result.
\end{remark}
We could have defined the Gauss-Manin connection and its residue analytically, essentially in the same way, obtaining \(((\mathbb R^i\pi_*\omega^{\bullet}_{{\mathfrak X}^{\times}/S^{\times}})^{an},\nabla^{an})\). We choose a small open disc \(\Delta\) arounf \(P\), then the restriction \(((\mathbb R^i\pi_*\omega^{\bullet}_{{\mathfrak X}^{\times}/S^{\times}})^{an}\mid_{\Delta},\nabla^{an}\mid_{\Delta})\) is a logarithmic extension of the Gauss-Manin connection \((\R^i {\pi'}_*(\Omega^{\bullet}_{{\mathfrak X'}/\Delta^*}),\nabla))\), notation as in the beginning of the section. Thanks to Remark \ref{ev0}, we are in the hypotheses of \cite[Th\'eor\`eme II.1.17 and Corollaire II.5.6]{Deligne70} so we deduce the following corollary.
\begin{corollary}
The monodromy endomorphisms \(T_s\) on \(\HH^1(\mathfrak X_s^{an};\C)\simeq \HH^1_{dR}(\mathfrak X_s/\C)\), defined above, are the fibers of an automorphism \(T\) of the locally free sheaf \(\R^1 \pi_*(\omega^{\bullet}_{{\mathfrak X}^{\times}/S^{\times}})\) whose fiber in \(P\) is
\[
T_P=\exp(-2\pi i \Res_P(\nabla)).
\]
If one fixes an identification
\begin{equation}
\HH^1(\mathfrak X_s^{an};\C)\simeq \HH^1_{dR}(\mathfrak X_s/\C)\overset{\sim}{\longrightarrow}\HHH^1(X,\omega^\bullet_{X/k^{\times}}),
\end{equation}
usinig Proposition \ref{locfreesten}, then \(T_P\) and \(T_s\) are conjugate.
\end{corollary}
In view of the previous corollary, the operator \(N\) is the classical nilpotent monodromy endomorphism viewed on \(\HHH^1(X,\omega^{\bullet}_{X/k^{\times}})\).

\section{Comparison of the monodromies}\label{comp}
In the previous section, we have described the classical monodromy operator \(N\) acting on \(\HHH^1(X_{Zar},\omega^{\bullet}_{{X_{Zar}}/{k^{\times}}_{Zar}})\), where the complex of log differentials was constructed using the Zariski topology. Now we come back at considering sheaves with respect to the \'etale topology, so that \(\omega^1_{X/k^{\times}}\) will denote the sheaf of log differentials on the \'etale site, as in §\ref{res}. Recall that, from Remark \ref{zaret}, we have
\[
\omega^{\bullet}_{{X_{Zar}}/{k^{\times}}_{Zar}}\simeq \varepsilon_*\omega^\bullet_{X/k^{\times}},
\]
for \(\varepsilon :X_{\acute et}\to X_{Zar}\) the morphism of sites. In order to compare the monodromy \(N\) with the one constructed in \ref{mon}, we will show that the natural map
\[
\HHH^p(X_{Zar},\varepsilon_*\omega^\bullet_{X/k^{\times}})\longrightarrow \HHH^p(X_{\acute et},\omega^{\bullet}_{X/k^{\times}})=\HH^p_{\log}(X/k^{\times})
\]
is an isomorphism. Since the cohomologies above are the abutments of the spectral sequence associated to the stupid filtration on the complexes, it is sufficient to show that
\[
\HHH^p(X_{Zar},\varepsilon_*\omega^q_{X/k^{\times}})\overset{\sim}{\longrightarrow} \HHH^p(X_{\acute et},\omega^{q}_{X/k^{\times}}).
\]
One can prove this for any quasi-coherent sheaf \(\mathcal F\) on \(X_{\acute et}\) using the Leray spectral sequence
\[
\HHH^p(X_{Zar},\R^q\varepsilon_* \mathcal F)\Longrightarrow \HHH^{p+q}(X_{\acute et},\mathcal F).
\]
Indeed, computing the sections of \(\R^q\varepsilon_* \mathcal F\) on affine open sets, shows \(\R^q \varepsilon_* \mathcal F=0\) for \(q>0\) when \(\mathcal F\) is quasi-coherent (see \cite[Ch. VII, 4.2.2]{SGA4_2} for the details).\\
In light of this, we keep calling \(N\) the monodromy on \(\HH^1_{\log}(X/k^{\times})\). Proposition \ref{stenprop} implies that \(N\) is obtained from the connecting morphism in the short exact sequence
\begin{equation*}
0\longrightarrow f^* \omega^1_{k^{\times}}\otimes_{\mathcal O_{X}}\omega^{\bullet-1}_{X/k^{\times}}\overset{\wedge}{\longrightarrow}\omega^{\bullet}_{X}\longrightarrow\omega^{\bullet}_{X/k^{\times}}\longrightarrow 0,
\end{equation*}
the analogue of (\ref{sesfib}) on the \'etale site. In section §\ref{mon}, we have introduced also a combinatorial monodromy \(\tilde{N}\) on \(\HH^1_{\log}(X/k^{\times})\). The rest of this paper is devoted to show that the combinatorial monodromy \(\tilde{N}\) and the classical monodromy \(N\) coincide. Namely, we prove the following.
\begin{theorem}\label{moncomp}
With the previous notation, \(\tilde{N}=N\).
\end{theorem}
Using the quasi isomorphism
\[
\rho:\omega^\bullet_{X/k^\times}\longrightarrow\Tot(\mathcal A_{X/k^\times}),
\]
defined in (\ref{qisotot}), we will reduce the proof to a comparison of endomorphisms on \(\Tot(\mathcal A_{X/k^\times})\) in \(D^b(X)\). For this we take an absolute version of the double complex \(\mathcal A_{X/k^{\times}}\), namely
\[
\mathcal A_{X/k}=
\begin{tikzcd}
\bigoplus\limits_{v\in\mathscr V}i^{\circ}_{v*}\omega^2_{X_v/k} \arrow[r,"\eta^2"] &\bigoplus\limits_{e\in\mathscr E}i^{\circ}_{e*}\omega^2_{X_e}\\
\bigoplus\limits_{v\in\mathscr V}i^{\circ}_{v*}\omega^1_{X_v/k} \arrow[r,"\eta^1"] \arrow[u,"d=0"]&\bigoplus\limits_{e\in\mathscr E}i^{\circ}_{e*}\omega^1_{X_e}\arrow[u,"d=0"]\\
\bigoplus\limits_{v\in\mathscr V}i^{\circ}_{v*}\mathcal O_{X_v} \arrow[r,"\eta^0"]\arrow[u,"d"] &\bigoplus\limits_{e\in\mathscr E}i^{\circ}_{e*}\mathcal O_{X_e}\arrow[u,"d=0"]\nospaceperiod
\end{tikzcd}
\]
One can show, as in lemma \ref{mv}, that the obvious map \(\omega_{X/k}^\bullet\to\Tot(\mathcal A_{X/k})\) is a quasi-isomorphism.

Using the computations of log differentials in §\ref{logcurve}, one shows that we have short exact sequences analogue to (\ref{sesfib})
\begin{align*}
0\longrightarrow f^* \omega^1_{k^{\times}}\otimes_{\mathcal O_{X}}i^{\circ}_{v*}\omega^{\bullet-1}_{X_v/k^{\times}}\overset{\wedge}{\longrightarrow}i^{\circ}_{v*}\omega^{\bullet}_{X_v/k}\longrightarrow i^{\circ}_{v*}\omega^{\bullet}_{X_v/k^{\times}}\longrightarrow 0, \\
0\longrightarrow f^* \omega^1_{k^{\times}}\otimes_{\mathcal O_{X}}i^{\circ}_{e*}\omega^{\bullet-1}_{X_e/k^{\times}}\overset{\wedge}{\longrightarrow}i^{\circ}_{e*}\omega^{\bullet}_{X_e/k}\longrightarrow i^{\circ}_{e*}\omega^{\bullet}_{X_e/k^{\times}}\longrightarrow 0,
\end{align*}
for each vertex \(v\) in \(\mathscr V\) and edge \(e\) in \(\mathscr E\). Putting those together via direct sum, one obtains a short exact sequence
\begin{equation}\label{sestot}
0\longrightarrow\mathcal \Tot(A_{X/k^\times})^{\bullet-1}\overset{\varphi}{\longrightarrow}\Tot(A_{X/k})\overset{\psi}{\longrightarrow} \Tot(A_{X/k^\times})\longrightarrow 0.
\end{equation}
This in turn gives a connecting morphism \(\Tot(A_{X/k^\times})\to\Tot(A_{X/k^\times})^{\bullet-1}[1]\) in the derived category of bounded complexes of abelian sheaves \(D^b(X)\). After alternating the signs to make the differentials agree, i.e. composing with \((-1)^i\id:\Tot(A_{X/k^\times})^{\bullet-1}[1]\to\Tot(A_{X/k^\times})^{\bullet}\), we get
\begin{equation}\label{nu}
\nu:\Tot(A_{X/k^\times})\longrightarrow\Tot(A_{X/k^\times}).
\end{equation}
We want to point out that the morphism $\psi$ has an index-wise section $\sigma$ (is not a cochain map) described locally as follows. First we consider the case of an \'etale open around a double point \(X_e\in X_v\cap X_w\). On this \'etale open, \(\Tot^0(\mathcal A_{X/k^\times})\simeq i^{\circ}_{v*}\mathcal  O_{X_v}\oplus i^{\circ}_{w*}\mathcal  O_{X_w}\) so we put
\begin{align*}
\sigma^0:\Tot^0(\mathcal A_{X/k^\times})&\longrightarrow \Tot^0(\mathcal A_{X/k})\\
\left(f_v,f_w\right)&\longmapsto\left(f_v,f_w\right).
\end{align*}
In degree \(1\), we have \(\Tot^1(\mathcal A_{X/k^\times})\simeq i^{\circ}_{v*}\omega^1_{X_v/k^\times}\oplus i^{\circ}_{w*}\omega^1_{X_w/k^\times}\oplus i^{\circ}_{e*} \mathcal O_{X_e}\). Remember, from §\ref{logcurve}, that a section of \(\omega^1_{X_v/k^\times}\) is given by a class \([f_vd\log x+g_vd\log y]\) modulo the relation \(d\log x+d\log y=0\). So, a local section of \(\Tot^1(\mathcal A_{X/k^\times})\) has a representative of the form \(([f_vd\log x],[g_w d\log y], a_e)\). We define
\begin{align*}
\sigma^1:\Tot^1(\mathcal A_{X/k^\times})&\longrightarrow \Tot^1(\mathcal A_{X/k})\\
([f_vd\log x],[g_w d\log y], a_e)&\longmapsto (f_vd\log x, g_w d\log y, a_e).
\end{align*}
Finally, in degree \(2\), a section of \(\Tot^2(\mathcal A_{X/k^\times})\) is a section of \(\omega^1_{X_e/k^\times}\) (again expressed as an equivalence class) while a section of \(\Tot^2(\mathcal A_{X/k})\) is a section of \(i^{\circ}_{v*}\omega^2_{X_v/k}\oplus i^{\circ}_{w*}\omega^2_{X_w/k}\oplus i^\circ_{e*}\omega^1_{X_e/k}\). We declare
\begin{align*}
\sigma^2:\Tot^2(\mathcal A_{X/k^\times})&\longrightarrow \Tot^2(\mathcal A_{X/k})\\
[a_ed\log x + b_ed\log y]&\longmapsto (0,0,(b_e-a_e)d\log y).
\end{align*}
The map \(\sigma\) around smooth points is similar:
\begin{align*}
\sigma^0:i^{\circ}_{v*}\mathcal  O_{X_v}&\longrightarrow i^{\circ}_{v*}\mathcal  O_{X_v}\\
f_v&\longmapsto f_v\\
\sigma^1:i^{\circ}_{v*}\omega^1_{X_v/k^\times}&\longrightarrow i^{\circ}_{v*}\omega^1_{X_v/k}\\
[f_vdz+g_v d\log 0]&\longmapsto f_vdz\\
\sigma^2:0&\longrightarrow i^{\circ}_v\omega^2_{X_v/k}.
\end{align*}
\begin{lemma}\label{compconnmorph}
The connecting morphisms \(\partial\), defined in (\ref{connmorphpartial}), and \(\nu\) defined in (\ref{nu}), coincide via the quasi isomorphism \(\rho\), that is
\[
\partial=\rho^{-1}\circ\nu\circ\rho.
\]
\end{lemma}
\begin{proof}
One only needs to verify that \(\rho\) is compatible with the formation of the short exact sequences (\ref{sesfib}) and (\ref{sestot}), but this is straightforward. 
\end{proof}
We want to show that also \(\tilde{N}\) comes from some morphism on \(\Tot(\mathcal A_{X/k^\times})\). So we consider the composition
\begin{equation}\label{conntot}
\Tot(\mathcal A_{X/k^\times})\longrightarrow\bigoplus_{v\in \mathscr V}i^{\circ}_v\omega_{X_v}^\bullet \overset{\eta_A}{\longrightarrow}\bigoplus_{e\in \mathscr E}i^{\circ}_e\omega_{X_e}^\bullet \overset{\Res}{\longrightarrow}\bigoplus_{e\in \mathscr E}i^{\circ}_e\omega_{X_e}^{\bullet-1}\longrightarrow \Tot(\mathcal A_{X/k^\times}),
\end{equation}
where the unlabeled morphisms are, respectively, the projection to the first column of the double complex and the inclusion of the second column of the double complex. The morphism \(\eta_A\) was defined in (\ref{etaA}). We will denote the composition of the morphisms in (\ref{conntot}) with \(\tilde{\nu}\).\\
Then \(\tilde{\nu}\) computes the combinatorial monodromy \(\tilde{N}\) in the following sense.
\begin{lemma}\label{numon}
The endomorphism induced by \(\rho^{-1}\circ \tilde{\nu} \circ \rho\) in hypercohomology and \(\tilde{N}\) coincide.
\end{lemma}
\begin{proof}
This comes directly from the description of \(\HH^1_{\log}(X/k^{\times})\) in terms of global sections in Lemma \ref{expl}.
\end{proof}
To summarize, we have seen in Lemma \ref{numon} that
\[
\tilde{N}=\R^1f_*(\rho^{-1}\circ\tilde{\nu}\circ\rho),
\]
while Proposition \ref{stenprop} and Lemma \ref{compconnmorph} yield
\[
N=\R^1f_*(\partial)=\R^1f_*(\rho^{-1}\circ\nu\circ\rho).
\]
We then deduce Theorem \ref{moncomp} showing that \(\tilde{\nu}=\nu\).
\begin{lemma}
The morphisms \(\tilde{\nu}:\Tot(A_{X/k^\times})\to\Tot(A_{X/k^\times})\) and \(\nu:\Tot(A_{X/k^\times})\to\Tot(A_{X/k^\times})\), described above, coincide in \(D^b(X)\).
\end{lemma}

\begin{proof}
We will consider the mapping cone \(\Cone (\varphi)\). For an index \(i\), we have \(\Cone(\varphi)^i=\Tot(A_{X/k})^i\oplus\Tot(A_{X/k^\times})^i\) and differential \(d_{\Cone(\varphi)}^i(b,a)=(d^i (b)+\varphi(a), -d^i (a))\). Then the morphism \(\nu\) is described by the roof, called (left)-``fraction'' in \cite[§10.5.3]{Weibel94},
\[
\nu:
\begin{tikzcd}
&\Cone(\varphi) \arrow[dl,"s"'] \arrow[dr,"\delta"] &\\
\Tot(A_{X/k^\times}) & &\Tot(A_{X/k^\times})\nospacecomma
\end{tikzcd}
\]
where \(s^i(b,a)=\psi^i(b)\) and \(\delta^i(b,a)=(-1)^ia\), for any index \(i\). On the other hand, the morphism \(\tilde{\nu}\) is given by
\[
\tilde{\nu}:
\begin{tikzcd}
&\Tot(A_{X/k^\times}) \arrow[dl,"\id"'] \arrow[dr,"\tilde{\nu}"] &\\
\Tot(A_{X/k^\times}) & &\Tot(A_{X/k^\times})\nospacecomma
\end{tikzcd}
\]
The two roofs coincide once we exhibit a commutative diagram
\[
\begin{tikzcd}
&&\Xi \arrow[dl,"\lambda"']\arrow[dr,"\mu"]&&\\
&\Tot(A_{X/k^\times}) \arrow[dl,"\id"']\arrow[drrr,"\tilde{\nu}"]&&\Cone(\varphi)\arrow[dlll,"s"'] \arrow[dr,"\delta"] \\
\Tot(A_{X/k^\times})&&&&\Tot(A_{X/k^\times})\nospacecomma
\end{tikzcd}
\]
where \(\lambda\) is a quasi isomorphism of complexes and \(\mu\) is a morphism of complexes in the homotopical category. We will put \(\Xi=\Tot(\mathcal A_{X/k^\times})\), \(\lambda=\id\) and \(\mu=(\sigma,(-1)^i\tilde{\nu})\). Then the commutativity of the diagram is obvious once we show that \(\mu=(\sigma,(-1)^i\tilde{\nu})\) is a cochain map. The statement being local, we will verify it (\'etale locally) first around a double point. For \(i=0\), we have
\[
d_{\Cone(\varphi)}^0(\mu(f_v,f_w))=d_{\Cone(\varphi)}^0((f_v,f_w),0)=((df_v,df_w,f_v|_{X_e}-f_w|_{X_e}),0)
\]
and
\[
\mu(d^0(f_v,f_w))=\mu((df_v,df_w,f_v|_{X_e}-f_w|_{X_e}),0)=((df_v,df_w,f_v|_{X_e}-f_w|_{X_e}),0).
\]
For \(i=1\), we have
\begin{align*}
&d_{\Cone(\varphi)}^1\left(\mu([f_v d\log x], [g_w d\log y], a_e)\right)=d_{\Cone(\varphi)}^1((f_vd\log x,g_w d\log y,a_e),-f_v|_{X_e})\\
&=\left(\left(0,0,f_v|_{X_e}d\log x - g_w|_{X_e}d\log y\right)-\left(0,0,f_v|_{X_e}(d\log x +d\log y)\right),0\right)\\
&=((0,0,-(f_v|_{X_e}+g_w|_{X_e})d\log y),0)
\end{align*}
and
\begin{align*}
\mu(d^1([f_v d\log x], [g_w d\log y], a_e))&=\mu([f_v|_{X_e}d\log x - g_w|_{X_e}d\log y])\\
&=((0,0,-(f_v|_{X_e}+g_w|_{X_e})d\log y),0).
\end{align*}
For \(i=2\), we have
\begin{align*}
&d_{\Cone(\varphi)}^2(\mu([f_ed\log x]))=d_{\Cone(\varphi)}^2(-f_e d\log y,0)=(0,0)
\end{align*}
in agreement with the fact that \(d^2=0\) on \(\Tot(\mathcal A_{X/k^\times})^2\). The same verification around smooth points is direct.
\end{proof}

\bibliographystyle{abbrv}
\bibliography{bibliography}

\end{document}